\documentclass[12pt, reqno]{amsart}
\usepackage{amsmath, amsthm, amscd, amsfonts, amssymb, graphicx, color}
\usepackage[bookmarksnumbered, colorlinks, plainpages]{hyperref}
\hypersetup{colorlinks=true,linkcolor=red, anchorcolor=green, citecolor=cyan, urlcolor=red, filecolor=magenta, pdftoolbar=true}

\renewcommand{\phi}{\varphi}
\newtheorem{theorem}{Theorem}[section]
\newtheorem{lemma}[theorem]{Lemma}

\newtheorem{proposition}[theorem]{Proposition}
\newtheorem{corollary}[theorem]{Corollary}
\theoremstyle{definition}

\newtheorem{example}[theorem]{Example}

\theoremstyle{remark}

\numberwithin{equation}{section}

\begin{document}

\setcounter{page}{1}

\title[  On some algebras  of THO and ATHO]{On some algebras  of truncated Hankel  and Asymmetric truncated Hankel operators.}

\author[Ameur Yagoub and Muhammad Ahsan Khan]{Ameur Yagoub $^1$, Muhammad Ahsan Khan $^2$}

\address{$^{1}$ Laboratoire de math\'ematiques pures et appliqu\'es. Universit\'e de Amar Telidji. Laghouat, 03000. Algeria.}
\email{\textcolor[rgb]{0.00,0.00,0.84}{a.yagoub@lagh-univ.dz}}
\address{$^{2}$  Department of Mathematics University of Kotli Azad Jammu and
		Kashmir, Pakistan.}
\email{\textcolor[rgb]{0.00,0.00,0.84}{ahsankhan388@hotmail.com}}


\let\thefootnote\relax\footnote{Copyright 2016 by the Tusi Mathematical Research Group.}



\date{Received: xxxxxx; Accepted: zzzzzz.
\newline \indent $^{*}$Corresponding author
\newline $^\diamond$ Advance publication -- final volume, issue, and page numbers to be assigned.}

\begin{abstract}
	In the last decade, a large amount of research has been concentrated on the operators living on the model space. Asymmetric truncated Toeplitz operators and asymmetric truncated  Hankel operators are the natural generalization of truncated Toeplitz operators and truncated Hankel operators respectively. In this paper, we obtained the basic results concerning the product of these operators and in terms of product their connection with each other.  In addition, when the inner function has a certain symmetric property, some algebraic properties of truncated Hankel operators are also discussed.
\end{abstract} \maketitle
\section{Introduction}
Let $H^2$ be the classical  Hardy space of the open unit disk $\mathbb{D}$, consists of functions analytic in $\mathbb{D}$ with square summable Maclaurin
coefficients, and let $L^2:=L^2(\mathbb{T})$ be the space of all square-integrable functions on the one dimensional torus  $\mathbb{T}$ in the complex plane $\mathbb{C}$. The space $L^\infty\subset L^2$ is formed by essentially bounded functions on $\mathbb{T}$, while  $H^\infty\subset H^2$ is the algebra of bounded analytic functions on $\mathbb{D}$. The  inner function is an element  $ {u}\in H^2$ such that  $| u(z)|=1$ a.e. on $\mathbb{T}$. The model space corresponding to $u$ is the
closed subspace $K_ {u}$ of $H^2$ of the form $$K_{ {u}} := H^2 \ominus  {u} H^2.$$
It is a reproducing kernel Hilbert space with the reproducing kernels for points $\lambda \in \mathbb{D}$ are 
\begin{displaymath}
k_{\lambda}^{ {u}}(z)
		:= \frac{1 - \overline{ {u}(\lambda)}  {u}(z)}{1 - \overline{\lambda} z}, \quad z \in \mathbb{D}.
\end{displaymath}
There is a natural conjugation (a conjugate-linear isometric involution, see \cite{gm}) on $K_{ {u}}$
defined in terms of boundary functions by
$$C_ {u} f(z) :=\widetilde{f}(z)= \overline{ f(z) z} {u(z)},\quad |z|=1.$$ A simple computation shows that $\widetilde{k_{\lambda}^{ {u}}}(z) = \frac{ {u}(z) -  {u}(\lambda)}{z - \lambda}$. 
It is easy to verify that the map $  U : L^2 \rightarrow L^2$ given by
$$Uf(z) = \widehat{f}(z) = \overline{f(\overline{z})},\quad |z|= 1,$$
is a conjugation on $L^2$. 
One can also easily verify that for an inner function $u$, the function  $\widehat{u}(z) = \overline{u(\overline{z})}$ is also inner. The  inner function $u$ is called real symmetric if $u =\widehat{u}$.
The authors of \cite{CG} proved that $f \in K_u$ if and only if $\widehat{f}\in K_{\widehat{u}}$, and we have $\widehat{fg}=\widehat{f}\widehat{g}$, for every $f,g\in L^2$.
The function $ {u}$ is said to have an angular derivative in the sense of
Caratheodory ($ADC$) at the point $\eta\in\mathbb{T}$  if $ {u}$ has a nontangential limit
$ {u}(\eta)$ of unit modulus at $\eta$ and $ {u}'$ has a nontangential limit $ {u}'(\eta)$ at $\eta$, in particular  $k_{\eta}^{ {u}}(z) =\overline{ {u}(\eta)}\eta\widetilde{k_{\eta}^{ {u}}}(z)\in K_ {u}$ (see \cite{gm} for more details).\\
 Let $P_ {u} = P- {u} P \overline{ {u}}$ denotes the projection from $L^2$ to $K_ {u}$, where $P$ is the projection of $L^2$ on $H^2$. Let $S$ be the unilateral shift on $H^2$ and, for a nonconstant inner function $  {u} $, let $S_ {u}= P_ {u} S |K_ {u}$ be the compression of $S$ to $K_ {u}$ with its adjoint $S_ {u}^*$ given by
\begin{displaymath}
(S_ {u}^*f)(z)= \frac{f(z) - f(0)}{z } \quad for \quad z\in \mathbb{T}.
\end{displaymath}
For
each $\alpha\in\overline{\mathbb{D}}$, one can consider the following rank-one perturbation of $S_u$ on  $K_u$:
 $$
S_u^\alpha=S_u+\frac{\alpha}{1-\alpha\overline{u(0)}}k_0^u\otimes\widetilde{k_0^u}.
$$
Throughout in this paper, we will make use of the defect operators $I-S_ {u} S_ {u}^*=k_0^ {u}\otimes k_0^ {u}$ and $I-S^*_ {u} S_ {u}=\widetilde{k_0^ {u}}\otimes \widetilde{k_0^ {u}}$, where $f\otimes g$ is the rank one operator that maps $h$ to $f <h, g>$, see \cite{sar}. If $ \phi\in L^2 $, then the compression of the multiplication operator $ M_\phi $ to $ H^2 $ is called a \emph{Toeplitz operator} and is denoted by $ T_\phi $. That means that $ T_\phi=PM_\phi|H^2 $. More than a decade ago,
Sarason has introduced in~\cite{sar} the so-called \emph{truncated Toeplitz operators} (TTOs). For $\varphi\in L^2$, the TTO on $K_ {u}$  with symbol $\phi$ is the operator $A^u_\phi$ defined by
\begin{displaymath}
A_\varphi^{ {u}}f := P_ {u}(\varphi f ),\quad f\in K_ {u}.
\end{displaymath}
The authors in \cite{cj} and \cite{cj2,r} initiated the study
of so-called \emph{asymmetric truncated Toeplitz operators} (ATTOs). Let $u,v$ be two inner functions. An ATTO $ A_\varphi^{ {u},v} $
with symbol $\phi\in L^2$, is an operator from $K_u$ into $K_v$ densely defined by
\begin{displaymath}
A_\varphi^{ {u},v}f:= P_ {v}(\varphi f ),\quad f\in K_ {u}.
\end{displaymath}
Clearly, $A_\varphi^{ {u},u}=A_\varphi^{ {u}}$.  We write $ \varphi\overset{\mathrm{A}}{\equiv} \psi$ to mean that $A^{ {u},v}_\varphi = A^{ {u},v}_\psi$.
 Most recently the \emph{truncated Hankel operators} (THOs) were introduced by C. Gu in \cite{GU}. The THO on $K_ {u}$  with symbol $\phi$ is the operator $B^u_\phi$ defined as 
$$B^u_\phi f:= P_u J(I-P)(\phi f),\quad f\in K_u,$$
where $ J : L^2 \rightarrow L^2$ is the Flip operator given by,
$Jf(z)= \overline{z}f(\overline{z}), |z| = 1.$
It is notable that $J$ is a unitary operator which maps $H^2$ precisely  onto $\overline{zH^2}$ and $\overline{zH^2}$ onto $H^2$. 

Their asymmetric versions were introduced in \cite{LM}. An \emph{asymmetric truncated Hankel operator} (ATHO) $B^{u,v}_\phi$ with symbol $\phi\in L^2$ is an operator
from $K_u$ into $K_v$ densely defined as
$$B^{u,v}_\phi f:= P_v J(I-P)(\phi f),\quad f\in K_u.$$ 
We write $ \varphi\overset{\mathrm{B}}{\equiv} \psi$ to mean that $B^{ {u},v}_\varphi = B^{ {u},v}_\psi$. 
If $K_u$ and $K_v$ are model spaces corresponding to inner functions $u$ and $v$ respectively, we will use the following notations:
 \begin{itemize}
 	 \item $\mathcal{T}(u)$ is the space of all bounded TTOs on
 	$K_ {u}$.
 	\item  $\mathcal{T}(u,v)$ is the space of all bounded ATTOs from $K_u$ into $K_v$.
 	
 	\item   $\mathcal{H}(u)$ is the space of all bounded THOs on $K_ {u}$.
 	\item   $\mathcal{H}(u,v)$ is the space of all bounded ATHOs from $K_ {u}$ into $K_v$.
 \end{itemize} 
It is well known that the product of two TTOs is
not a TTO. In \cite{sed}, Sedlock has investigated when a product of TTOs is itself a TTOs and as a result completely identified  all the maximal commutative algebras in $\mathcal{T}( {u})$. The author in \cite{y} obtained some results concerning the products  of ATTOs. D.O. Kanga and  H. J. Kimb \cite{KK} studied the  products  of THOs  when the inner function $u$ has a certain symmetric property.\\
 The purpose of the current paper is  to obtain the basic  results concerning the products of TTOs, ATTOs, THOs and  ATHOs , apart from this when the  inner function $u$ has a certain symmetric property, some of the algebraic properties of THOs are also obtained .\\
 The structure of the paper is the following. By means of section ~1 and section ~2 we want to make sure that the reader become acquainted with basic notations and other useful
 facts from this area, needed when we are going to start the main work in next sections. After the section of  preliminaries in section ~3 we will obtain the product relation for ATHOs and ATTOs. In section ~4 we obtained some algebraic properties of THOs. The last part of the paper, comprising section ~5 and section ~6, obtained results related to the products of THOs and ATHOs.   
\section{Preliminaries.}

As stated in section ~1 that  Sedlock completely characterized  the maximal commutative algebras in $\mathcal{T}( {u})$.
These algebras $\mathcal{B}_{ {u}}^\alpha$, where the parameter $\alpha$ belongs to the Riemann sphere
$\alpha\in \widehat{\mathbb{C}}=\mathbb{C}\cup\{\infty\}$, are
$$\mathcal{B}_{ {u}}^\alpha :=\left\{A^{ {u}}_{\phi + \alpha \overline{S_ {u} C_ {u}\varphi} +c} \in \mathcal{T}( {u}): \phi \in K_{ {u}},c\in\mathbb{C}\right\},\textit{ and }
\mathcal{B}_{ {u}}^\infty :=\left\{\ A^{ {u}}_{ \overline{\varphi}}\in \mathcal{T}( {u}): \phi \in K_{ {u}}\right\}.$$
In the following Lemma from \cite{sed} , we summarize the main properties of the classes $\mathcal{B}_{ {u}}^\alpha$ and their links to $S_u^\alpha$.
\begin{lemma}\label{pro} Let $u$ be an inner function.
		
		\begin{enumerate}
		\item [(1)]$ A\in\mathcal{B}_{u}^{\alpha} $ if and only if $A^*\in\mathcal{B}_{u}^{1/\overline{\alpha}}\ $.

			\item[(2)] If $|\alpha| = 1$, then $S_{u}^{\alpha} $ is a unitary operator  and $$\mathcal{B}_{u}^{\alpha} = \{S_{u}^{\alpha}\}'=\left\lbrace \Phi(S_u^\alpha),\Phi\in L^\infty(\mu_\alpha)\right\rbrace, $$ 
		where $\mu_\alpha$ is the Clark measure (see \cite{CL}). Each operator in $\mathcal{B}_{u}^{\alpha} $ is unitarily equivalent to a multiplication operator $M_\Phi$ on $L^2(\mu_\alpha)$ induced by a function $\Phi\in L^\infty(\mu_\alpha)$.	
			\item [(3)] If $|\alpha| < 1$, then $S_u^\alpha$ is a completely nonunitary contraction, and 
			$$\mathcal{B}_{u}^{\alpha} = \{S_{u}^{\alpha}\}'=\left\lbrace \Psi(S_u^\alpha)=A^u_{\frac{\Psi}{1-\alpha\overline{u}}},\Psi\in H^\infty\right\rbrace. $$ 
			\item[(4)] If $|\alpha| > 1$, then $$\mathcal{B}_{u}^{\alpha} = \{(S_u^{1/\overline{\alpha}})^*\}'=\left\lbrace \widehat{\Psi}(S_u^{1/\overline{\alpha}})^*= A^u_{\frac{\alpha\overline{\Psi}}{\alpha-u}},\Psi\in H^\infty\right\rbrace.$$
			\item[(5)] If $A,B\in\mathcal{T}(u)$ , then $AB\in\mathcal{T}(u)$ if and only if either one of the operator in the product is a scalar multiple of an identity operator, or both belong to the same class $\mathcal{B}_u^ \alpha$ for some $\alpha\in\widehat{\mathbb{C}}$. In the last case we also have $AB\in\mathcal{B}_u^\alpha$.
			
		\end{enumerate}
	\end{lemma}
Remember that an operator $A$  is said to be $C_u$-symmetric if $C_uAC_u = A^*$. We quote few results from \cite{cj,cj2,GP,JL,jl,l,120} concerning  ATTOs.
\begin{lemma}\label{lemma}
Let $ {u},v$ be two inner functions. 
\begin{enumerate}
\item[(1)] $A_\varphi^{ {u},v}= 0$ if and only if $\phi\in v H^2+\overline{ {u} H^2}$.

\item[(2)] The bounded linear operator $A$ from $K_u$ into $K_v$, belongs to $\mathcal{T}( {u},v)$ if and only if there are functions $\phi_1\in K_v$ and $\phi_2\in K_ {u}$ such that $$A-S_v AS^*_ {u}=\phi_1\otimes k_0^ {u}+k_0^v\otimes \phi_2.$$ In this case, we have $A=A^{ {u},v}_{\phi_1+\overline{\phi_2}}$.
\item[(3)] $$A^{ {u},v}_{k_0^v}=A^{ {u},v}_{1} =A^{ {u},v}_{\overline{k_0^ {u}}}.$$
\item[(4)] The operators in $\mathcal{T}( {u})$ are $C_u$-symmetric.
\item[(5)] If $ \lambda\in\mathbb{D}, $ then the operators $\widetilde{k^v_\lambda }\otimes k_\lambda^ {u}$ and $k^v_\lambda \otimes \widetilde{k_\lambda^ {u}}$ belong to $ \mathcal{T}( {u},v)$, and we have
$$\widetilde{k^v_\lambda }\otimes k_\lambda^ {u}=A^{ {u},v}_{\frac{v}{z-\lambda}}\quad\text{and}\quad k^v_\lambda \otimes \widetilde{k_\lambda^ {u}}=A^{ {u},v}_{\frac{\overline{ {u}}}{\overline{z-\lambda}}}.$$
 \item[(6)] If both $  {u} $ and $ v $ have an $ADC$ at a point $ \eta $ of $ \mathbb{T}$, then the operator $k^v_\eta \otimes k_\eta^ {u}$ belongs to $ \mathcal{T}( {u},v)$, and  $$k^v_\eta \otimes k_\eta^ {u}=A^{ {u},v}_{k^v_\eta+\overline{k_\eta^ {u}}-1}.$$
  \end{enumerate}
\end{lemma}

The following result from \cite{y} give  us the necessary and sufficient condition for the product of two ATTOs to be an ATTO.  
\begin{lemma}\label{yag} Let $ {u},v$ be two inner functions. Let $A\in\mathcal{T}(v,w)$ and $B\in\mathcal{T}(u,v)$ be such that $A = A^{v,w}_{\varphi_1+\overline{\psi_1}}$ and $B = A^{u,v}_{\varphi_2+\overline{\psi_2}}$, for some $\varphi_1\in K_w$, $\psi_1,\phi_2\in K_v$, and $ \psi_2\in K_u$. Then $AB\in\mathcal{T}(u,w)$ if and only if
	$$ (\phi_1 \otimes \psi_2)-(P_w v(\phi_1+\overline{\psi_1}) \otimes P_u v(\psi_2+\overline{\phi_2})  )= (\Phi\otimes k^u_0)+(k^w_0\otimes\Psi),$$
	for some $\Phi\in K_w$ and $\Psi\in K_u$.
	\end{lemma}

We gather in the next Lemma some basic facts about ATHOs, that can be found in \cite{ry,KK,GU} (see also \cite{L2,LM}).
\begin{lemma}\label{lem24} Let $ {u},v$ be two inner functions and  $\phi\in L^2$.
\begin{enumerate}
\item[(1)]$(B^{u,v}_\phi)^* = B^{v,u}_{\widehat{\phi}}.$ 
\item[(2)] $B^{u,v}_\phi = 0$ if and only if $\phi\in H^2 +\overline{u\widehat{v}H^2}$.
\item[(3)]For each $B \in\mathcal{H}( {u,v})$ there exists   $\psi\in K_{u\widehat{v}}$ such that $B = B^u_{\overline{\psi}} $.
\item[(4)]A bounded linear operator $B \in\mathcal{H}( {u,v})$ if and only if there exist $ \phi_1\in K_{v},\phi_2\in K_u$ such
that $$B-S_v^* BS^*_ {u} =\phi_1\otimes k_0^ {u}+\widetilde{k^v_0}\otimes \varphi_2.$$

\item[(5)] If $ \lambda\in\mathbb{D} $ then the operators $\widetilde{k^v_{\overline{\lambda} }}\otimes \widetilde{k^u_\lambda }$ and $k^v_{\overline{\lambda} } \otimes k_\lambda^ {u}$ belong to $ \mathcal{H}( {u},v)$, and, If $  {u} $ and $ v $ have an ADC at $\eta$ and $ \overline{\eta} $, respectively, then the operators  $\widetilde{k^v_{\overline{\eta} }}\otimes \widetilde{k^u_\eta }$ and $k^v_{\overline{\eta} } \otimes k_\eta^ {u}$ also  belong to $ \mathcal{H}( {u},v).$
\end{enumerate}
\end{lemma}
We end this section with  the following simple observations.
\begin{proposition}Let $ {u}$ be an inner function and $\lambda\in\mathbb{D}$. Then
\begin{enumerate} 
\item[(1)]  $U k_\lambda^u  =k_{\overline{\lambda}}^{\widehat{u}}$.
\item[(2)] $U \widetilde{k_\lambda^u}  =\widetilde{k_{\overline{\lambda}}^{\widehat{u}}}$.

\end{enumerate}
\end{proposition}
\begin{proof}We have \begin{enumerate} 
\item[(1)] $U k_\lambda^u(z)  =  \frac{1 -  {u}(\lambda)  \widehat{u}(z)}{1 - \lambda z}=\frac{1 - \overline{ \widehat{u}(\overline{\lambda})}  \widehat{u}(z)}{1 - \overline{\overline{\lambda}} z}=k_{\overline{\lambda}}^{\widehat{u}}.$
\item[(2)]$U \widetilde{k_\lambda^u}  =U(\frac{u(z)-u(\lambda)}{z-\lambda})=\frac{\widehat{u}-\overline{u(\lambda)}}{z-\overline{\lambda}}=\frac{\widehat{u}-\widehat{u}(\overline{\lambda})}{z-\overline{\lambda}}=\widetilde{k_{\overline{\lambda}}^{\widehat{u}}}$.

\end{enumerate}
\end{proof}
\section{Relation between  asymmetric truncated Hankel and Toeplitz operators.}\label{31}
Authors in \cite{GU2} have given the connection between ATHOs and ATTOs, which enables us to transfer the known results for ATTOs to their Hankel analogues.  We start with the following results. 
\begin{lemma}\cite{GU2}\label{Lam1}	Let $u,v$ be two inner functions and  $\phi\in L^2$. Then 
\begin{enumerate}
\item $C_v A_\phi^{u,v} C_u =A^{u,v}_{\overline{u}v\overline{\phi}}$,
\item $U A_\phi^{u,v}U =A^{\widehat{u},\widehat{v}}_{\widehat{\phi}}$,
\item $C_v B_\phi^{u,v} C_u = B^{u,v}_{\overline{u\widehat{v}\phi}}$,
\item $U B_\phi^{u,v} U = B^{\widehat{u},\widehat{v}}_{\widehat{\phi}}$,
\item $ C_v A_\phi^{u,v} U=B^{\widehat{u},v}_{\overline{\widehat{v}}\widehat{\phi}}$,
\item $ U A_\phi^{u,v} C_u=B^{u,\widehat{v}}_{\overline{u}\overline{\phi}}$,
\item $U B_\phi^{u,v} C_u= A^{u,\widehat{v}}_{\overline{u\phi}}$,
\item $ C_vB_\phi^{u,v} U=A^{\widehat{u},v}_{v\widehat{\phi}}$.
\end{enumerate}
\end{lemma}

\begin{lemma}\cite{GU2}\label{Lam2}
Let $u,v$ be two inner functions $A$ and $B$ be bounded linear operators from $K_u$ into $K_v$. Then
\begin{enumerate}
\item[(1)]$A\in\mathcal{T}( {u,v})$ if and only if $C_vAC_u\in\mathcal{T}( {u},v)$,
\item[(2)] $B\in\mathcal{H}( {u},v)$ if and only if $C_vBC_u\in\mathcal{H}( {u},v)$,
\item[(3)]$A\in\mathcal{T}( {u},v)$ if and only if $UAU\in\mathcal{T}( {\widehat{u}},\widehat{v})$,
\item[(4)]$B\in\mathcal{H}( {u},v)$ if and only if $UBU\in\mathcal{H}( {\widehat{u}},\widehat{v})$,
\item[(5)]$A\in\mathcal{T}( {u},v)$ if and only if $C_vAU\in\mathcal{H}( {\widehat{u}},v)$,
\item[(6)]$B\in\mathcal{H}( {u},v)$ if and only if $UBC\in\mathcal{T}( {u,\widehat{v}})$.
\end{enumerate}
\end{lemma}
The following corollary gathers some consequences of Lemma 3.1 and Lemma 3.2.
\begin{corollary}Let $ {u}$ be an inner function and $\alpha\in\overline{\mathbb{D}}.$ Then	we have,
\begin{enumerate}
\item[(1)]  $U \mathcal{B}_{ {u}}^\alpha U  =\mathcal{B}_{ \widehat{u}}^{\overline{\alpha}}$ for every $\alpha\in \widehat{\mathbb{C}}$. 
\item [(2)]$ C_u U=B^{\widehat{u},u}_{\overline{\widehat{u}}}$, and $U C_u=B^{u,\widehat{u}}_{\overline{u}}.$ Moreover, $C_{\widehat{u}}U=UC_u$ and $C_uU=UC_{\widehat{u}}$.
\item [(3)]$U S_{\widehat{u}}^\alpha U =S_{u}^{\overline{\alpha}}$.
\item [(4)]$C_u S_u^\alpha C_u =(S_{u}^\alpha)^*$.
\end{enumerate}
\end{corollary}
\begin{proof} \item[(1)] Let $\alpha\in \mathbb{C}$ and  $A\in \mathcal{B}_{ {u}}^\alpha ,$ then  $A=A^{ {u}}_{\phi + \alpha \overline{S_ {u} C_ {u}\varphi} +c}$ with  $\phi \in K_{ {u}}$ and $c\in\mathbb{C}$, it follows  then from Lemma \ref{Lam1} that 
$$ UA^{ {u}}_{\phi + \alpha \overline{S_ {u} C_ {u}\varphi} +c}U=A^{ \widehat{u}}_{U(\phi + \alpha \overline{S_ {u} C_ {u}\varphi} +c)}=
A^{ \widehat{u}}_{\widehat{\phi} + \overline{\alpha}\overline{S_ {\widehat{u}} C_ {\widehat{u}}\widehat{\varphi}} +\overline{c}}\in\mathcal{B}_{ \widehat{u}}^{\overline{\alpha}}. $$
If $\alpha=\infty$, we have 
 $
A\in \mathcal{B}_{ {u}}^\infty $, then  $A = A^{ {u}}_{ \overline{\varphi}}, \phi \in K_{ {u}},$ and  $U A^{ {u}}_{ \overline{\varphi}}U=A^{ \widehat{u}}_{ \overline{\widehat{\varphi}}} \in\mathcal{B}_{ \widehat{u}}^{\infty}.$ \\
\item [(2)] Applying  Lemma \ref{Lam1} ((5) and  (6)) to the case $\phi=1$, we get,  $C_{\widehat{u}}U=B^{\widehat{u},u}_{\overline{\widehat{u}}}=UC_u$ and $C_uU=B^{\widehat{u},u}_{\overline{\widehat{u}}}=UC_{\widehat{u}}$.

\item [(3)]$U S_u^\alpha U =US_uU+\frac{\overline{\alpha}}{1-u(0)\overline{\alpha}}Uk_0^u\otimes U\widetilde{k_0^u}=S_{\widehat{u}}+\frac{\overline{\alpha}}{1-{\overline{\widehat{u}(0)}}\overline{\alpha}}k_0^{\widehat{u}}\otimes \widetilde{k_0^{\widehat{u}}}=S_{\widehat{u}}^{\overline{\alpha}}.$
\item [(4)]We have  $S_u^\alpha\in\mathcal{T}( {u})$, then $S_u^\alpha$ is $C_u$-symmetric, therefore $C_u S_u^\alpha C_u =(S_{u}^\alpha)^*$.

\end{proof}

\begin{corollary}\label{remark}Let $ {u},v$ be two  inner functions and $\phi\in L^2$. Then 
\item [(1)] $ C_{\widehat{v}}UB_\phi^{u,v} =A^{u,\widehat{v}}_{\widehat{v}\phi}$.
\item [(2)] $ B_\phi^{u,v}C_{u}U =A^{\widehat{u},v}_{\overline{\widehat{u}\widehat{\phi}}}$.
\item [(3)] If $u$ is real symmetric and  $D=D^*=C_{u}U=B^{u}_{\overline{u}}$,  then 
 $ DB_\phi^{u} =A^{u}_{u\phi}$,
, $ B_\phi^{u}D =A^{u}_{\overline{u\widehat{\phi}}}$
and $ DS_u^\alpha =(S_u^{\overline{\alpha}})^* D$.
\end{corollary}

In the rest of this section we give the relations concerning the products of ATHOs and ATTOs. 
\begin{proposition}\label{1} Let $u$, $v$ and $w$ be nonconstant inner functions, and  $\phi_1, \phi_2\in L^2 $, then  the following are equivalent:
\begin{enumerate}
\item[(1)]$ B^{ {v},w}_{\phi_1}B^{ {u},v}_{\phi_2}\in\mathcal{H}( {u},w)$.
\item[(2)] $A^{\widehat{v},w}_{w\widehat{\phi_1}}A^{u,\widehat{v}}_{\overline{u\phi_2}}\in\mathcal{H}( {u},w)$ 
\item[(3)]$A^{v,\widehat{w}}_{\overline{v\phi_1}}A^{\widehat{u},v}_{v\widehat{\phi_2}}\in\mathcal{H}( \widehat{u},\widehat{w})$. 
\item[(4)]$ A^{v,\widehat{w}}_{\overline{v\phi_1}}B^{ {u},v}_{\overline{u\widehat{v}\phi_2}}\in\mathcal{T}( {u,\widehat{w}})$. 
\end{enumerate}
\end{proposition}
\begin{proof} Let  $\phi_1, \phi_2\in L^2 $, then Lemma \ref{Lam1} and Lemma \ref{Lam2} imply that  \\
$(1)\Leftrightarrow(2)$
\begin{eqnarray*}
B^{ {v},w}_{\phi_1}B^{ {u},v}_{\phi_2}\in\mathcal{H}( {u},w)&\Leftrightarrow& C_wC_wB^{ {v},w}_{\phi_1}UUB^{ {u},v}_{\phi_2}C_uC_u\in\mathcal{H}( {u},w)\\
&\Leftrightarrow&  C_wA^{\widehat{v},w}_{w\widehat{\phi_1}}A^{u,\widehat{v}}_{\overline{u\phi_2}}C_u\in\mathcal{H}( {u},w)\\
&\Leftrightarrow& A^{\widehat{v},w}_{w\widehat{\phi_1}}A^{u,\widehat{v}}_{\overline{u\phi_2}}\in\mathcal{H}( {u},w).
\end{eqnarray*}
$(1)\Leftrightarrow (3)$
\begin{eqnarray*}
B^{ {v},w}_{\phi_1}B^{ {u},v}_{\phi_2}\in\mathcal{H}( {u},w)&\Leftrightarrow& UUB^{ {v},w}_{\phi_1}C_vC_vB^{ {u},v}_{\phi_2}UU\in\mathcal{H}( {u},w)\\
&\Leftrightarrow&  UA^{v,\widehat{w}}_{\overline{v\phi_1}}A^{\widehat{u},v}_{v\widehat{\phi_2}}U\in\mathcal{H}( {u},w)\\
&\Leftrightarrow&A^{v,\widehat{w}}_{\overline{v\phi_1}}A^{\widehat{u},v}_{v\widehat{\phi_2}}\in\mathcal{H}( \widehat{u},\widehat{w}).
\end{eqnarray*}
$(1)\Leftrightarrow (4)$
 \begin{eqnarray*}
B^{ {v},w}_{\phi_1}B^{ {u},v}_{\phi_2}\in\mathcal{H}( {u},w)&\Leftrightarrow& UB^{ {v},w}_{\phi_1}C_uC_uB^{ {u},v}_{\phi_2}C_u\in\mathcal{T}( {u,\widehat{w}})\\
&\Leftrightarrow&  A^{v,\widehat{w}}_{\overline{v\phi_1}}B^{ {u},v}_{\overline{u\widehat{v}\phi_2}}\in\mathcal{T}( {u,\widehat{w}}).
\end{eqnarray*}

\end{proof}

\begin{corollary} \label{3} Let $u$, $v$ and $w$ be nonconstant inner functions, and $\phi_1, \phi_2\in L^2 $, then  the following are equivalent:
\begin{enumerate}
\item[(1)]$ B^{ {v},w}_{\phi_1}B^{ {u},v}_{\phi_2}\in\mathcal{T}( {u},w)$.
\item[(2)] $A^{\widehat{v},w}_{w\widehat{\phi_1}}A^{u,\widehat{v}}_{\overline{u\phi_2}}\in\mathcal{T}( {u},w)$.
\item[(3)]$A^{v,\widehat{w}}_{\overline{v\phi_1}}A^{\widehat{u},v}_{v\widehat{\phi_2}}\in\mathcal{T}( \widehat{u},\widehat{w})$. 

\item[(4)]$B^{ {v},w}_{\overline{v\widehat{w}\phi_1}} A^{\widehat{u},v}_{v\widehat{\phi_2}}\in\mathcal{H}( {\widehat{u}},w)$. 
\end{enumerate}
\end{corollary}
\begin{proof}The proof of $(1)\Leftrightarrow (2)$ and $(1)\Leftrightarrow (3)$ are similar to the proof of proposition \ref{1}. For $(1)\Leftrightarrow (4)$, we  have 
\begin{eqnarray*}
B^{ {v},w}_{\phi_1}B^{ {u},v}_{\phi_2}\in\mathcal{T}( {u},w)&\Leftrightarrow& C_wB^{v, {w}}_{\phi_1}B^{ {u},v}_{\phi_2}U\in\mathcal{H}( {\widehat{u}},w)\\
&\Leftrightarrow& B^{ {v},w}_{\overline{v\widehat{w}\phi_1}} A^{\widehat{u},v}_{v\widehat{\phi_2}}\in\mathcal{H}( {\widehat{u}},w).
\end{eqnarray*}
\end{proof}

\section{Some algebraic properties of THOs .}\label{4}
Let $u$ be a real symmetric inner function and  $\phi\in L^2 $, we have   $$B^{ {u}}_{\phi}(B^{ {u}}_{\phi})^*=B^{ {u}}_{\phi}D(B^{ {u}}_{\phi}D)^*, \textit{ and }(B^{ {u}}_{\phi})^*B^{ {u}}_{\phi}=(DB^{ {u}}_{\phi})^*DB^{ {u}}_{\phi}.$$
It is proved in \cite{CT} that a TTO $A^{ {u}}_{\phi}$ is unitary if and only if it belongs to some class
$\mathcal{B}_u^\alpha$ for some $\alpha \in\mathbb{T}$. In this case $A^{ {u}}_{\phi} = \Phi(S_u^\alpha)$, where $|\Phi| = 1$ $\mu_\alpha$-almost everywhere. We give a similar result for THOs. 
\begin{theorem}Let $u$ be  a real symmetric. If $B^{ {u}}_{\phi} \in \mathcal{H}( {u})$, then the following are equivalent:
\begin{enumerate} 
\item[(1)] $B^{ {u}}_{\phi}$ is an isometry.
\item[(2)] $B^{ {u}}_{\phi}$ is a coisometry.
\item[(3)] $B^{ {u}}_{\phi}$ is unitary.
\item[(4)] $(B^{ {u}}_{\phi})^*B^{ {u}}_{\phi}-I\in\mathcal{T}( {u})$, 
\item[(5)] $B^{ {u}}_{\phi}(B^{ {u}}_{\phi})^*-I\in\mathcal{T}( {u})$,
\item[(6)]$B^{ {u}}_{\phi}\in D\mathcal{B}_u^\alpha$ for some $\alpha \in\mathbb{T}$, and $B^{ {u}}_{\phi} = D\Phi(S_u^\alpha)$, where $|\Phi| = 1$ $\mu_\alpha$-almost everywhere.
\end{enumerate}
\end{theorem}
\begin{proof}Since $DB^{ {u}}_{\phi}=A^{u}_{u\phi}\in \mathcal{T}( {u})$, by   \cite{CT} Theorem 6.3, we have $B^{ {u}}_{\phi}$ is an isometry $ \Leftrightarrow $  $(DB^{ {u}}_{\phi})^*DB^{ {u}}_{\phi}=I$ $ \Leftrightarrow $ $DB^{ {u}}_{\phi}$ is an isometry $ \Leftrightarrow $ $DB^{ {u}}_{\phi}$ is a coisometry $ \Leftrightarrow $ $DB^{ {u}}_{\phi}(DB^{ {u}}_{\phi})^*=I$
$\Leftrightarrow $ $B^{ {u}}_{\phi}(B^{ {u}}_{\phi})^*=I$ $ \Leftrightarrow $ $B^{ {u}}_{\phi}$ is a coisometry $ \Leftrightarrow $ $B^{ {u}}_{\phi}$ is unitary.\\
The equivalences  $(1)\Leftrightarrow (4)$, and  $(1)\Leftrightarrow (5)$,  are all
immediate.\\
 $DB^{ {u}}_{\phi}\in \mathcal{T}( {u})$ is unitary $ \Leftrightarrow $ $DB^{ {u}}_{\phi}\in \mathcal{B}_u^\alpha$ for some $\alpha \in\mathbb{T}$, and $DB^{ {u}}_{\phi} = \Phi(S_u^\alpha)$, where $|\Phi| = 1$ $\mu_\alpha$-almost everywhere.
\end{proof}
\begin{corollary}Let $ {u}$ be an inner function. If $B^{ {u},\widehat{u}}_{\phi} \in \mathcal{H}( {u},\widehat{u})$, then the following are equivalent:
\begin{enumerate} 
\item[(1)] $B^{ {u},\widehat{u}}_{\phi}$ is an isometry.
\item[(2)] $B^{ {u},\widehat{u}}_{\phi}$ is a coisometry.
\item[(3)] $(B^{ {u},\widehat{u}}_{\phi})^*B^{ {u},\widehat{u}}_{\phi}-I_{K_u}\in\mathcal{T}( {u})$, 
\item[(4)] $B^{ {u},\widehat{u}}_{\phi}(B^{ {u},\widehat{u}}_{\phi})^*-I_{K_{\widehat{u}}}\in\mathcal{T}( \widehat{u})$,
\item[(5)]$B^{ {u},\widehat{u}}_{\phi}\in UC_u\mathcal{B}_u^\alpha$ for some $\alpha \in\mathbb{T}$, and $B^{ {u}}_{\phi} = UC_u\Phi(S_u^\alpha)$, where $|\Phi| = 1$ $\mu_\alpha$-almost everywhere.
\end{enumerate}
\end{corollary}
\begin{proof}Since $UB^{ {u},\widehat{u}}_{\phi}C_u=A^{u}_{\overline{u\phi}}\in \mathcal{T}( {u})$, we have $(B^{ {u},\widehat{u}}_{\phi})^*B^{ {u},\widehat{u}}_{\phi}=I_{K_u}\Leftrightarrow (A^{u}_{\overline{u\phi}})^*A^{u}_{\overline{u\phi}}=I_{K_u} \Leftrightarrow A^{u}_{\overline{u\phi}}(A^{u}_{\overline{u\phi}})^*=I_{K_{u}}\Leftrightarrow
 B^{ {u},\widehat{u}}_{\phi}(B^{ {u},\widehat{u}}_{\phi})^*=I_{K_{\widehat{u}}}$ , rest of the proof is analogous to the proof of the preceding theorem. 
\end{proof}
The following result  is an analogue of Theorem 5.4. in \cite{sed}.
\begin{theorem}Let $u$ be a real symmetric and  $B^{ {u}}_{\phi} \in \mathcal{H}( {u})$ be invertible, then the following are equivalent:
\begin{enumerate} 
\item[(1)] $(B^{ {u}}_{\phi})^{-1}\in \mathcal{H}( {u})$,
\item[(2)]  $(DB^{ {u}}_{\phi})^{-1}\in\mathcal{T}( {u})$  
\item[(3)] $B^{ {u}}_{\phi}\in D\mathcal{B}_u^\alpha$ and $(B^{ {u}}_{\phi})^{-1}\in D\mathcal{B}_u^{\frac{1}{\alpha}}$ for some $\alpha\in\mathbb{C}\setminus\{0\}$.  
\item[(4)]  There are functions $\phi_1, \phi_2\in K_u$ and $c_1,c_2\in\mathbb{C}$
 such that $$B^{ {u}}_{\phi}=B^{ {u}}_{\overline{u}\phi_1 + \alpha \overline{uS_ {u} C_ {u}\varphi_1} +\overline{u}c_1} \textit{ and }(B^{ {u}}_{\phi})^{-1}=B^{ {u}}_{\overline{u}\phi_2 + \frac{1}{\alpha} \overline{uS_ {u} C_ {u}\varphi_2} +\overline{u}c_2},$$ for some $\alpha\in\mathbb{C}\setminus\{0\}$.
\end{enumerate}
\end{theorem}
\begin{proof}Let $B^{ {u}}_{\phi} \in \mathcal{H}( {u})$ be invertible, by \cite{sed}, we have $(B^{ {u}}_{\phi})^{-1}\in \mathcal{H}( {u})$ $ \Leftrightarrow $ $B^{ {u}}_{\phi}(B^{ {u}}_{\phi})^{-1}=(B^{ {u}}_{\phi})^{-1}B^{ {u}}_{\phi}=I$   $ \Leftrightarrow $ $(DB^{ {u}}_{\phi})^{-1}\in \mathcal{T}( {u})$  $ \Leftrightarrow $ there exists $\alpha\in\mathbb{C}\setminus\{0\}, DB^{ {u}}_{\phi}, (DB^{ {u}}_{\phi})^{-1}\in \mathcal{B}_u^\alpha$  $ \Leftrightarrow $   $B^{ {u}}_{\phi}\in D\mathcal{B}_u^\alpha$ and $(B^{ {u}}_{\phi})^{-1}\in \mathcal{B}_u^\alpha D=D\mathcal{B}_u^{\frac{1}{\alpha}} $ for some $\alpha\in\mathbb{C}\setminus\{0\}$.\\
Let  $DB^{ {u}}_{\phi}\in \mathcal{B}_u^\alpha$ and $D(B^{ {u}}_{\phi})^{-1}\in \mathcal{B}_u^{\frac{1}{\alpha}}$ for some $\alpha\in\mathbb{C}\setminus\{0\}$. Since  $ DB_\phi^{u} =A^{u}_{u\phi}$ if and only if there is function $\phi_1\in K_u$ and $c_1\in\mathbb{C}$ such that

\begin{eqnarray*}
 DB_\phi^{u} =A^{u}_{u\phi}=A^{ {u}}_{\phi_1 + \alpha \overline{S_ {u} C_ {u}\varphi_1} +c_1}&\Leftrightarrow& u\phi=\phi_1 + \alpha \overline{S_ {u} C_ {u}\varphi_1} +c_1+uH^2+\overline{uH^2},	\\
&\Leftrightarrow &  \phi=\overline{u}\phi_1 + \alpha \overline{uS_ {u} C_ {u}\varphi_1 }+\overline{u}c_1+H^2+\overline{u^2H^2},	\\
&\Leftrightarrow & 	 \varphi\overset{\mathrm{B}}{\equiv} \overline{u}\phi_1 + \alpha \overline{uS_ {u} C_ {u}\varphi_1 }+\overline{u}c_1,\\
&\Leftrightarrow & B_\phi^{u} =B^{ {u}}_{\overline{u}\phi_1 + \alpha \overline{uS_ {u} C_ {u}\varphi_1} +\overline{u}c_1},
\end{eqnarray*}
for some $\alpha\in\mathbb{C}\setminus\{0\}$.
For the symbol of $(B^{ {u}}_{\phi})^{-1}$, proof is similar.
\end{proof}
The following Theorem  is an analogue of Corollary 3.2 in \cite{CT}.
\begin{theorem} Let $u$ be  a real symmetric and $B^{ {u}}_{\phi_1},B^{ {u}}_{\phi_2} \in \mathcal{H}( {u})$, then the following are equivalent:
\begin{enumerate} 
\item[(1)] $B^{ {u}}_{\phi_1}B^{ {u}}_{\phi_2}=0$,  
\item[(2)] there is $\alpha\in\widehat{\mathbb{C}}$ such that $B^{ {u}}_{\phi_1}\in D\mathcal{B}_u^\alpha $ and $B^{ {u}}_{\phi_2}\in \mathcal{B}_u^\alpha D$  moreover: 
\begin{enumerate}
\item If $|\alpha| = 1$, then $B^{ {u}}_{\phi_1}=D\Phi(S_u^\alpha),B^{ {u}}_{\phi_2}=\Psi(S_u^\alpha)D$, with $\Phi,\Psi\in L^\infty(\mu_\alpha)$ and $\Phi\Psi=0, \mu_\alpha$-almost everywhere.
\item If $|\alpha| < 1$, then $B^{ {u}}_{\phi_1}=D\Phi(S_u^\alpha),B^{ {u}}_{\phi_2}=\Psi(S_u^\alpha)D$, with $\Phi,\Psi\in H^\infty$ and the inner function  $u_\alpha$ divides $\Phi\Psi$. 
\item If $|\alpha| > 1$, then  $B^{ {u}}_{\phi_1}=D\Phi(S_u^{1/\overline{\alpha}})^*,B^{ {u}}_{\phi_2}=\Psi(S_u^{1/\overline{\alpha}})^*D$, with $\Phi,\Psi\in H^\infty$  and the inner function  $u_{1/\overline{\alpha}}=\frac{1-\overline{\alpha}u}{u-\alpha}$ divides $\Phi\Psi$.
\end{enumerate}	\end{enumerate}
\end{theorem}
\begin{proof}The proof follows immediately from  Corollary 3.2 of \cite{CT},  for $B^{ {u}}_{\phi_1}B^{ {u}}_{\phi_2}=0\Leftrightarrow DB^{ {u}}_{\phi_1}B^{ {u}}_{\phi_2}D=0$.
\end{proof}
\begin{corollary}Let $ {u}$ be an inner function and  $B^{ {u},\widehat{u}}_{\phi_1} \in \mathcal{H}( {u},\widehat{u})$ and $B^{\widehat{u}, {u}}_{\phi_2} \in \mathcal{H}( \widehat{u},{u})$, then the following are equivalent:
\begin{enumerate} 
\item[(1)] $B^{\widehat{u}, {u}}_{\phi_1}B^{ {u},\widehat{u}}_{\phi_2}=0$,  
\item[(2)] There is $\alpha\in\widehat{\mathbb{C}}$ such that $B^{\widehat{u}, {u}}_{\phi_1}\in C_u\mathcal{B}_u^\alpha U$ and $B^{ {u},\widehat{u}}_{\phi_2}\in U\mathcal{B}_u^\alpha C_u$  moreover: 
\begin{enumerate}
\item If $|\alpha| = 1$, then $B^{\widehat{u}, {u}}_{\phi_1}=C_u\Phi(S_u^\alpha)U,B^{ {u},\widehat{u}}_{\phi_2}=U\Psi(S_u^\alpha)C_u$, with $\Phi,\Psi\in L^\infty(\mu_\alpha)$ and $\Phi\Psi=0, \mu_\alpha$-almost everywhere.
\item If $|\alpha| < 1$, then $B^{\widehat{u}, {u}}_{\phi_1}=C_u\Phi(S_u^\alpha)U,B^{ {u},\widehat{u}}_{\phi_2}=U\Psi(S_u^\alpha)C_u$, with $\Phi,\Psi\in H^\infty$ and the inner function  $u_\alpha$ divides $\Phi\Psi$. 
\item If $|\alpha| > 1$, then  $B^{\widehat{u}, {u}}_{\phi_1}=C_u\Phi(S_u^{1/\overline{\alpha}})^*U,B^{ {u},\widehat{u}}_{\phi_2}=U\Psi(S_u^{1/\overline{\alpha}})^*C_u$, with $\Phi,\Psi\in H^\infty$  and the inner function  $u_{1/\overline{\alpha}}=\frac{1-\overline{\alpha}u}{u-\alpha}$ divides $\Phi\Psi$.
\end{enumerate}	\end{enumerate}
\end{corollary}
\begin{proof} By Corollary \ref{3},  $B^{\widehat{u}, {u}}_{\phi_1}B^{ {u},\widehat{u}}_{\phi_2}=0\Leftrightarrow A^{u}_{u\widehat{\phi_1}}A^{u}_{\overline{u\phi_2}}=0$, if and only if by Corollary 3.2 of \cite{CT},  finishing the proof. 

\end{proof}
\section{Products  of truncated Hankel operators.}\label{5}
If $u$ is any real symmetric inner function then the  following result provides a necessary and sufficient condition for the product of two THOs to be a TTO. The result is  an analogue  of Theorem 2.7 of  \cite{KK}.
\begin{theorem} \label{th2} Let $u$ be  a real symmetric and  $B_1,B_2 \in \mathcal{H}( {u})$, then  $ B_1B_2 \in\mathcal{T}( {u})$ if and only if one of two cases
holds:
\begin{enumerate}
\item $B_1=cD$, or $B_2=cD,$ for some $c\in\mathbb{C}$.
\item $\exists \alpha\in \widehat{\mathbb{C}}$, $B_1(S_u^{\overline{\alpha}})^*= S_u^\alpha B_1$, and $B_2S_u^\alpha=(S_u^{\overline{\alpha}})^* B_2$. 
\end{enumerate}In
which case their product is in $\mathcal{B}_u^\alpha$.
\end{theorem}
\begin{proof}By Corollary \ref{remark}, $B_1D,DB_2\in\mathcal{T}( {u})$,  we have $ B_1B_2\in\mathcal{T}( {u})\Leftrightarrow B_1DDB_2\in\mathcal{T}( {u})$, if and only if by Lemma \ref{pro} (5),
\[\begin{cases}B_1D=cI,or, DB_2=cI,c\in\mathbb{C},&\\
or\\
\exists \alpha\in \widehat{\mathbb{C}},B_1DS_u^\alpha=(S_u^\alpha) B_1D,and,DB_2(S_u^\alpha)=S_u^\alpha DB_2.\end{cases}\]
\[\quad\quad\quad\quad\quad\Leftrightarrow\begin{cases}B_1=cD,or, B_2=cD,c\in\mathbb{C},&\\
or\\
\exists \alpha\in \widehat{\mathbb{C}},B_1(S_u^{\overline{\alpha}})^*= S_u^\alpha B_1,and,B_2S_u^\alpha=(S_u^{\overline{\alpha}})^* B_2.\end{cases}\]
If $ B_1B_2\in\mathcal{T}( {u})$, we have    $$B_1B_2S_u^\alpha=B_1(S_u^{\overline{\alpha}})^*B_2=S_u^\alpha B_1B_2,$$ 
then $ B_1B_2 $ commutes with $ S_u^\alpha $, therefore $ B_1B_2 \in\mathcal{B}_u^\alpha. $
\end{proof}
\begin{corollary}   Let $u$ be real symmetric and $B_1,B_2 \in \mathcal{H}( {u})$, then  $ B_1B_2 \in\mathcal{T}( {u})$ if and only if one of two cases
holds:
\begin{enumerate}
\item $B_1=cD$, or $B_2=cD,c\in\mathbb{C}$.
\item $\exists \alpha\in \widehat{\mathbb{C}}$, $B_1\in\mathcal{B}_{ {u}}^\alpha D$, and $B_2\in D\mathcal{B}_{ {u}}^\alpha $.
\end{enumerate}
\end{corollary}
\begin{proof}The proof follows immediately from Theorem (\ref{th2}) and Lemma \ref{pro} (5).

\end{proof}
We give here  an analogue of Theorem 2.9. of  \cite{KK}. 
\begin{theorem} Let $u$ be  a real symmetric and $ \phi_1,\phi_2 \in L^2$, with none of $B^{ {u}}_{\phi_1}$ and $B^{ {u}}_{\phi_2}$ constant multiple of $D$. Then $ B^{ {u}}_{\phi_1}B^{ {u}}_{\phi_2}\in\mathcal{T}( {u})$ if and only if there exists $\alpha\in \widehat{\mathbb{C}}$, and $\psi_1,\psi_2\in K_u$, such that  \[\begin{cases}\phi_1\overset{\mathrm{A}}{\equiv} \overline{{u}}\overline{\widehat{\psi_2}} + \alpha \overline{{u}}S_ {{u}} C_ {{u}}\widehat{\psi_1} +\overline{{u}}c_1,&\\
and\\
\phi_2\overset{\mathrm{A}}{\equiv} \overline{u}\psi_2 + \alpha \overline{u}\overline{S_ {u} C_ {u}\psi_2} +\overline{u}c_2.\end{cases}\]

\end{theorem}
\begin{proof} Corollary \ref{3} imply that $ B^{ {u}}_{\phi_1}B^{ {u}}_{\phi_2}=(B^{ {u}}_{\phi_1}D)(DB^{ {u}}_{\phi_2})\in\mathcal{T}( {u})$ if and only if  $A^{{u}}_{\overline{u\widehat{\phi_1}}}A^{{u}}_{u\phi_2}\in\mathcal{T}( {u})$,  if and only if there exists $\alpha\in \widehat{\mathbb{C}}$, such that $A^{{u}}_{\overline{u\widehat{\phi_1}}},A^{{u}}_{u\phi_2}\in\mathcal{B}_u^\alpha$  if and only if there exists $\alpha\in \widehat{\mathbb{C}}$, and $\psi_1,\psi_2\in K_u$, $A^{{u}}_{\overline{u\widehat{\phi_1}}}=A^{ {u}}_{\psi_1 + \alpha \overline{S_ {u} C_ {u}\psi_1} +c_1}$, and, $ A^{{u}}_{u\phi_2}=A^{ {u}}_{\psi_2 + \alpha \overline{S_ {u} C_ {u}\psi_2} +c_2}$
 \[\Leftrightarrow\begin{cases}\overline{u\widehat{\phi_1}}\overset{\mathrm{A}}{\equiv} \psi_1 + \alpha \overline{S_ {u} C_ {u}\psi_1} +c_1,&\\
and\\
u\phi_2\overset{\mathrm{A}}{\equiv} \psi_2 + \alpha \overline{S_ {u} C_ {u}\psi_2} +c_2.\end{cases}\Leftrightarrow
\begin{cases}\phi_1\overset{\mathrm{A}}{\equiv} \overline{{u}}\overline{\widehat{\psi_2}} + \alpha \overline{{u}}S_ {{u}} C_ {{u}}\widehat{\psi_1} +\overline{{u}}c_1,&\\
and\\
\phi_2\overset{\mathrm{A}}{\equiv} \overline{u}\psi_2 + \alpha \overline{u}\overline{S_ {u} C_ {u}\psi_2} +\overline{u}c_2.\end{cases}\]

\end{proof}
The below written results give us another characterization for symbol .
   \begin{theorem} Let $u$ be  a real symmetric and  $B_1,B_2 \in \mathcal{H}( {u})$, with none of $B_1$ and $B_2$ constant multiple of $D$. Then $ B_1B_2\in\mathcal{T}( {u})$ if and only if there is $\alpha\in\widehat{\mathbb{C}}$ such that $B_1\in \mathcal{B}_u^\alpha D $ and $B_2\in D\mathcal{B}_u^\alpha $,  moreover: 
\begin{enumerate}
\item If $|\alpha| = 1$, then $B_1=\Phi(S_u^\alpha)D,B_2=D\Psi(S_u^\alpha)$, with $\Phi,\Psi\in L^\infty(\mu_\alpha)$ and  $B_1B_2=\Phi\Psi(S_u^\alpha)$.
\item If $|\alpha| < 1$, then $B_1=A^u_{\frac{\Psi_1}{1-\alpha\overline{u}}}D=B^{ {u}}_{\frac{\overline{u\widehat{\Psi_1}}}{1-\alpha {u}}},B_2=DA^u_{\frac{\Psi_2}{1-\alpha\overline{u}}}=B^{ {u}}_{\frac{\overline{u}\Psi_2}{1-\alpha\overline{u}}}$, with $\Psi_1,\Psi_2\in H^\infty$  and  $B_1B_2=A^u_{\frac{\Psi_1\Psi_2}{1-\alpha\overline{u}}}$. 
\item If $|\alpha| > 1$, then  $B_1=A^u_{\frac{\alpha\overline{\Psi_1}}{\alpha-u}}D=B^u_{\frac{\alpha\overline{u}\widehat{\Psi_1}}{\alpha-\overline{u}}},B_2=DA^u_{\frac{\alpha\overline{\Psi_2}}{\alpha-u}}=B^u_{\frac{\alpha\overline{u}\overline{\Psi_2}}{\alpha-u}}$, with $\Psi_1,\Psi_2\in H^\infty$  and  $B_1B_2=A^u_{\frac{\alpha(\overline{\Psi_1\Psi_2})}{\alpha-u}}$.
\end{enumerate}
\end{theorem}
\begin{proof}This is proved directly in \cite{sed}, for $B_1B_2= B_1DDB_2\in\mathcal{T}( {u})$.

\end{proof}

The following result gives the answer to the question  that when is the  product of THO   and  TTO is  a THO subjectedto the condition that an  inner function $u$ is real symmetric. The result is  an analogue of Theorem 2.10 of \cite{KK}.

\begin{theorem} \label{5} Let $u$ be  a real symmetric with  $A\in\mathcal{T}( {u})$ and $B \in \mathcal{H}( {u})$. Then $AB\in\mathcal{H}(u)$ if and only if one of the following conditions holds:
\begin{enumerate}
\item $A=cI,$ or, $B=cD,c\in\mathbb{C}$,
\item $\exists \alpha\in \widehat{\mathbb{C}}$, such that $AS_u^\alpha= S_u^\alpha A,$ and, $B(S_u^{\overline{\alpha}})^*= S_u^\alpha B$.
\end{enumerate}
\end{theorem}
\begin{proof}Let $A\in\mathcal{T}( {u})$ and $B \in \mathcal{H}( {u})$,  we have $ AB\in\mathcal{H}( {u})\Leftrightarrow ABD\in\mathcal{T}( {u})$, if and only if by Lemma \ref{pro} (5), 
\begin{align*}\begin{cases}A=cI,or, BD=cI,c\in\mathbb{C},&\\
or\\
\exists \alpha\in \widehat{\mathbb{C}},AS_u^\alpha=S_u^\alpha A,and,BDS_u^\alpha=S_u^\alpha BD.\end{cases}\\
\quad\quad\quad\quad\quad\Leftrightarrow\begin{cases}A=cI,or, B=cD,c\in\mathbb{C},&\\
or\\
\exists \alpha\in \widehat{\mathbb{C}},AS_u^\alpha= S_u^\alpha A,and,B(S_u^{\overline{\alpha}})^*=S_u^\alpha B.\end{cases}
\end{align*}
\end{proof}
\begin{corollary}\label{coc20} Let $u$ be a real symmetric and  $A\in\mathcal{T}( {u})$ and $B \in \mathcal{H}( {u})$. Then $AB\in \mathcal{H}( {u})$  if and only if one of the following conditions holds:
\begin{enumerate}
\item $A=cI,$ or, $B=cC_u,c\in\mathbb{C}$,
\item $\exists \alpha\in \widehat{\mathbb{C}}$, such that $A\in\mathcal{B}_{ {u}}^\alpha$, and $B\in\mathcal{B}_{ {u}}^\alpha C_u$.
\end{enumerate}
\end{corollary}
\begin{proof} The proof follows immediately from  Theorem (\ref{5}) and  Lemma \ref{pro}(5).
\end{proof}
\begin{corollary} Let $u$ be a real symmetric and  $A\in\mathcal{T}( {u})$ and $B \in \mathcal{H}( {u})$. Then $BA\in\mathcal{H}(u)$ if and only if one of the following conditions holds:
\begin{enumerate}
\item $A=cI,$ or, $B=cD,c\in\mathbb{C}$,
\item $\exists \alpha\in \widehat{\mathbb{C}}$, such that $AS_u^\alpha= S_u^\alpha A,$ (or $A\in\mathcal{B}_{ {u}}^\alpha),$ and $BS_u^\alpha=(S_u^{\overline{\alpha}})^* B$ (or $B\in D\mathcal{B}_{ {u}}^\alpha $).
\end{enumerate}
\end{corollary}
\begin{proof} The proof is similar to the proof of the Theorem (\ref{5}) and Corollary \ref{coc20}.
\end{proof}
Now, we give some examples of products of rank-one  truncated Hankel operators.
\begin{example}
Let $ \lambda\in\mathbb{D} $, and $u$ be a real symmetric. Consider the rank one  truncated Hankel and Toeplitz operators $B_1=\widetilde{k_\lambda^u} \otimes \widetilde{k_{\overline{\lambda}}^u}, B_2=k^u_{\overline{\lambda }}\otimes k_\lambda^u,$ and $A=\widetilde{k_\lambda^u} \otimes k_{\lambda}^u$. Then
\item [(1)] 
   $B_1B_2=(\widetilde{k_\lambda^u} \otimes \widetilde{k_{\overline{\lambda}}^u})(k^u_{\overline{\lambda }}\otimes k_\lambda^u)=\overline{u'(\overline{\lambda})}(\widetilde{k_\lambda^u}  \otimes k_{{\lambda}}^u )\in\mathcal{T}(u).$
By Theorem \ref{th2}, $B_1D=(\widetilde{k_\lambda^u} \otimes \widetilde{k_{\overline{\lambda}}^u}) D=(\widetilde{k_\lambda^u} \otimes k_{\lambda}^u)$, and by Example 5.3. of  \cite{sed}, we have $(\widetilde{k_\lambda^u} \otimes k_{\lambda}^u)\in\mathcal{B}_{u}^{u(\lambda)}.$ While,
 $DB_2 =D(k^u_{\overline{\lambda }}\otimes k_\lambda^u)=(\widetilde{k_\lambda^u} \otimes k_{\lambda}^u)\in\mathcal{B}_{u}^{u(\lambda)}.$
\item [(2)] $AB_1=(\widetilde{k_\lambda^u} \otimes k_{\lambda}^u)(\widetilde{k_\lambda^u} \otimes \widetilde{k_{\overline{\lambda}}^u})=u'(\lambda)(\widetilde{k_\lambda^u} \otimes \widetilde{k_{\overline{\lambda}}^u})\in\mathcal{H}(u).$ By Theorem \ref{5}, $A=(\widetilde{k_\lambda^u} \otimes k_{\lambda}^u)\in \mathcal{B}_{u}^{u(\lambda)}$, and  $B_1D =(\widetilde{k_\lambda^u} \otimes k_{\lambda}^u)\in\mathcal{B}_{u}^{u(\lambda)}.$
\end{example}

\section{Products  of asymmetric truncated Hankel operators.}\label{6}
The following result provides a necessary and sufficient condition for the product of two ATHOs to be an ATTO..
\begin{theorem} \label{th} Let $ \phi_1\in \overline{K_{v\widehat{w}}} ,\phi_2 \in \overline{K_{u\widehat{v}}}$, then  $ B^{ {v},w}_{\phi_1}B^{ {u},v}_{\phi_2}\in\mathcal{T}( {u},w)$ if and only if,
$$\overline{\widehat{v}\widehat{\phi_1}}\otimes\overline{\widehat{v}\phi_2}-P_w\overline{\widehat{\phi_1}}\otimes P_ {u} \overline{\phi_2}= (\Phi\otimes k^ {u}_0)+(k^w_0\otimes\Psi),$$
	for some $\Phi\in K_u$ and $\Psi\in K_ {u}$.

\end{theorem}
\begin{proof}Let  $ \phi_1\in \overline{K_{v\widehat{w}}}$ and $\phi_2 \in \overline{K_{u\widehat{v}}}$, by Corollary \ref{3}, we have $ B^{ {v},w}_{\phi_1}B^{ {u},v}_{\phi_2}\in\mathcal{T}( {u},w)\Leftrightarrow A^{\widehat{v},w}_{\overline{\widehat{v}\widehat{\phi_1}}}A^{u,\widehat{v}}_{\widehat{v}\phi_2}\in\mathcal{T}( {u},w)$, since
	$\overline{\widehat{v}\widehat{\phi_1}}\in K_{w}$ and
	$\widehat{v}\phi_2\in K_{u}$, by Lemma \ref{yag}, we have,  $A^{\widehat{v},w}_{\overline{\widehat{v}\widehat{\phi_1}}+0}A^{u,\widehat{v}}_{0+\overline{\overline{\widehat{v}\phi_2}}}\in\mathcal{T}( {u},w)$    if and only if
	$$
	\overline{\widehat{v}\widehat{\phi_1}}\otimes\overline{\widehat{v}\phi_2}-P_w \widehat{v}\overline{\widehat{v}\widehat{\phi_1}}\otimes P_ {u} \widehat{v}\overline{\widehat{v}\phi_2}= (\Phi\otimes k^ {u}_0)+(k^w_0\otimes\Psi),
	$$if and only if
		$$\overline{\widehat{v}\widehat{\phi_1}}\otimes\overline{\widehat{v}\phi_2}-P_w\overline{\widehat{\phi_1}}\otimes P_ {u} \overline{\phi_2}= (\Phi\otimes k^ {u}_0)+(k^w_0\otimes\Psi),$$
	for some $\Phi\in K_w$ and $\Psi\in K_ {u}$.
\end{proof}
	\begin{corollary}\label{co6} Let $ \phi_1\in \overline{K_{v\widehat{u}}} ,\phi_2 \in \overline{K_{u\widehat{v}}}$, then  $ B^{ {v},u}_{\phi_1}B^{ {u},v}_{\phi_2}\in\mathcal{T}( {u})$ if and only if,
$$\overline{\widehat{v}\widehat{\phi_1}}\otimes\overline{\widehat{v}\phi_2}-P_u\overline{\widehat{\phi_1}}\otimes P_ {u} \overline{\phi_2}= (\Phi\otimes k^ {u}_0)+(k^u_0\otimes\Psi),$$
	for some $\Phi,\Psi\in K_ {u}$.
\end{corollary}	
\begin{proof} Applying Theorem \ref{th} to the case $u=w$.
\end{proof}
\begin{corollary} Let $ \phi\in \overline{K_{v\widehat{u}}} $, then  $ B^{ {v},u}_{\phi}(B^{ {v},u}_{\phi})^*\in\mathcal{T}( {u})$ if and only if,
$$\overline{\widehat{v}\widehat{\phi}}\otimes\overline{\widehat{v}\widehat{\phi}}-P_u\overline{\widehat{\phi}}\otimes P_ {u} \overline{\widehat{\phi}}= (\Phi\otimes k^ {u}_0)+(k^u_0\otimes\Phi),$$
	for some $\Phi\in K_u$.
\end{corollary}	
\begin{proof} Applying Corollary \ref{co6} to the case $\phi_2=\widehat{\phi}$.
\end{proof} 
\begin{theorem} \label{th3} Let  $ \phi\in \overline{K_{v\widehat{w}}} ,\psi_1 \in K_{v}$ and $\psi_2 \in K_{u}$. If $B^{ {v},w}_{\phi}\in\mathcal{H}( v,{w})$ and $A^{ {u},v}_{\psi_1+\overline{\psi_2}} \in \mathcal{T}( {u},v)$ then, $ B^{ {v},w}_{\phi}A^{ {u},v}_{\psi_1+\overline{\psi_2}}\in\mathcal{H}( {u},w)$ if and only if
	$$
	\overline{v\phi}\otimes \overline{v}u\psi_1-P_ {\widehat{w}} \overline{\phi}\otimes S_u\widetilde{\psi_2}= (\Phi\otimes k^{u}_0)+(k^{\widehat{w}}_0\otimes\Psi),
	$$
		for some $\Phi\in K_{\widehat{w}}$ and $\Psi\in K_{u}$.

\end{theorem}
\begin{proof}Let  $ \phi\in \overline{K_{v\widehat{w}}} ,\psi_1 \in K_{v}$ and $\psi_2 \in K_{u}$, by Corollary \ref{3}, $ B^{ {v},w}_{\phi}A^{ {u},v}_{\psi_1+\overline{\psi_2}}\in\mathcal{H}( {u},w)\Leftrightarrow U B^{ {v},w}_{\phi}C_vC_vA^{ {u},v}_{\psi_1+\overline{\psi_2}}C_u\in\mathcal{T}( {u},\widehat{w})\Leftrightarrow  A^{v,\widehat{w}}_{\overline{v\phi}}A^{u,v}_{\overline{u}v\psi_2+\overline{\overline{v}u\psi_1}}\in\mathcal{T}( u,\widehat{w})$, since $\overline{v\phi}\in K_{\widehat{w}}$, $\overline{u}v\psi_2\in K_v$ and $\overline{v}u\psi_1\in K_u$ , by Lemma \ref{yag}, we have,  $A^{v,\widehat{w}}_{\overline{v\phi}+\overline{0}}A^{u,v}_{\overline{u}v\psi_2+\overline{\overline{v}u\psi_1}}\in\mathcal{T}( u,\widehat{w})$  if and only if
	$$
	\overline{v\phi}\otimes \overline{v}u\psi_1-P_ {\widehat{w}} v\overline{v\phi}\otimes P_ {u} v(\overline{v}u\psi_1+\overline{\overline{u}v\psi_2})= (\Phi\otimes k^{u}_0)+(k^{\widehat{w}}_0\otimes\Psi),
	$$
	for some $\Phi\in K_{\widehat{w}}$ and $\Psi\in K_{u}$, if and only if
	$$
	\overline{v\phi}\otimes \overline{v}u\psi_1-P_ {\widehat{w}} \overline{\phi}\otimes S_u\widetilde{\psi_2}= (\Phi\otimes k^{u}_0)+(k^{\widehat{w}}_0\otimes\Psi),
	$$
\end{proof}
\begin{corollary} Let  $ \phi\in \overline{K_{u\widehat{v}}} ,\psi_1 \in K_{w}$ and $\psi_2 \in K_{v}$. If $B^{ {u},v}_{\phi}\in\mathcal{H}( u,{v})$ and $A^{ {v},w}_{\psi_1+\overline{\psi_2}} \in \mathcal{T}( {v},w)$, then $ A^{ {v},w}_{\psi_1+\overline{\psi_2}}B^{ {u},v}_{\phi}\in\mathcal{H}( {u},w)$ if and only if
	$$
	\overline{v\widehat{\phi}}\otimes \overline{v}w\psi_2-P_ {\widehat{u}} \overline{\widehat{\phi}}\otimes S_w\widetilde{\psi_1}= (\Phi\otimes k^{w}_0)+(k^{\widehat{u}}_0\otimes\Psi),
	$$
		for some $\Phi\in K_{\widehat{u}}$ and $\Psi\in K_{w}$.
		\end{corollary}	
\begin{proof}We have $ (A^{ {v},w}_{\psi_1+\overline{\psi_2}}B^{ {u},v}_{\phi})^*\in\mathcal{H}( {w},u)$ if and only if $ B^{ {v},u}_{\widehat{\phi}}A^{ {w},v}_{\psi_2+\overline{\psi_1}}\in\mathcal{H}( {w},u)$ , then by Theorem \ref{th3},  we have
	$$
	\overline{v\widehat{\phi}}\otimes \overline{v}w\psi_2-P_ {\widehat{u}} \overline{\widehat{\phi}}\otimes S_w\widetilde{\psi_1}= (\Phi\otimes k^{w}_0)+(k^{\widehat{u}}_0\otimes\Psi),
	$$
		for some $\Phi\in K_{\widehat{u}}$ and $\Psi\in K_{w}$.
\end{proof} 

\section*{Acknowledgments} 
This research work
is supported by the General Direction of Scientific Research and Technological
Development (DGRSDT), Algeria.
\bibliographystyle{amsplain}

\end{document}